\documentclass[12pt]{article}
\pdfoutput=1
\usepackage[utf8]{inputenc}
\usepackage{tabularx}
\usepackage{geometry}
\usepackage{amsmath}
\usepackage{graphicx}
\usepackage{amsthm}
\usepackage{amsfonts}
\usepackage{amssymb}
\usepackage{placeins}
\usepackage{amsmath}
\usepackage{amssymb}
\usepackage{amsfonts}
\usepackage{amsthm}
\usepackage{comment}
\usepackage{mathrsfs}
\usepackage{pgfplots}
\usepackage{color}
\usepackage{cite}
\usepackage{url}
\usepackage{indent first}
\usepackage[english]{babel}
\usepackage[utf8]{inputenc}
\usepackage{hyperref}
\usepackage{setspace}
\usepackage{multicol}
\usepackage{float}
\newtheorem{theorem}{Theorem}[section]
\theoremstyle{definition}
\newtheorem{definition}[theorem]{Definition}

\newtheorem{lemma}[theorem]{Lemma}

\newtheorem{conjecture}{\bf Conjecture}
\usepackage{hyperref} \hypersetup{ colorlinks=true, linkcolor=blue, citecolor = black, filecolor=magenta,
urlcolor=blue, }
\usepackage[square,numbers]{natbib}

\newtheorem{proposition}{\bf Proposition}

\begin{document}

\title{Detecting Causality with Symplectic Quandles}
\author{Ayush Jain}
\date{August 2023}

\author{\noindent Ayush Jain\\
\small{The Shri Ram School Aravali, India}\\
\small{\texttt{ayushj2007@gmail.com}}}

\vskip 1em

\date{\today}

\maketitle

\begin{abstract}
\bigskip
We investigate the capability of Symplectic quandles to detect causality for (2+1)-dimensional globally hyperbolic spacetimes (X). Allen and Swenberg showed that Alexander-Conway polynomial is insufficient to distinguish connected sum of two Hopf links from the links in the family of Allen-Swenberg 2-sky like links, suggesting that it can not always detect causality in X. We find that symplectic quandles, combined with Alexander-Conway polynomial, can distinguish these two type of links, thereby suggesting their ability to detect causality in X. The fact that symplectic quandles can capture causality in the Allen-Swenberg example is intriguing since the theorem of Chernov and Nemirovski, which states that Legendrian linking equals causality, is proved using Contact Geometry methods.\\

\end{abstract}

%\onehalfspacing

\section{Introduction}
\emph{Knot} is a smooth embedding of a circle $S^{1}$ into a 3 dimension Euclidean space, $\mathbb{R}^3$. Two knots are considered equivalent if they can be transformed into each other using \emph{Reidemeister moves}. The simplest knot, which consists of a plain circle in the space, is known as the \emph{trivial knot} or an \emph{unknot}. When one or more knots are placed into $\mathbb{R}^3$, we obtain a \emph{link}. \emph{Hopf link} is the simplest non-trivial link.\\

A \emph{knot invariant} is a function that always assigns the same values to equivalent knots or links however different knots may have the same value. Knot (or link) invariants are important structures in the study of causally related points in a globally hyperbolic spacetimes. The \emph{Low Conjecture} \citep{Low88} posits that two events in a (2+1)-dimensional globally hyperbolic spacetimes are causally related if their skies are topologically linked. Natàrio and Tod \citep{Natario04} extended the conjecture to higher dimensional spacetimes and proposed the \emph{Legendrian Low Conjecture}, which claims that two points in a $(3+1)$-dimensional globally hyperbolic spacetime with a Cauchy surface $\mathbb{R}^3$ are causally related if their skies are Legendrian Linked. Chernov and Nemirovski \citep{CN10} proved both these conjectures, demonstrating that the relationship between causality and linking holds if the associated Cauchy surface $\Sigma$ is not coverable by $S^2$ or $S^3$. These results open up the possibility that link invariants can detect the causality of two events in the skies. Natàrio and Tod \citep{Natario04} paired several causally related events and discovered that the link associated with each pair has a non-trivial Kauffman polynomial that can easily be derived from Jones polynomial of a single variable. Nevertheless, their results are confined to a specific class of skies. Chernov, Martin and Petkova \citep{CMP20} indicated that more informative link invariants, such as \emph{Khovanov homology} (a "categorification" of the Jones Polynomial) and \emph{Heegard Floer homology} (a "categorification" of the Alexander-Conway polynomial), can detect causality in (2+1)-dimensional globally hyperbolic spacetime with $\Sigma \neq S^2, \mathbb{R} P^2$. However, the polynomial invariants themselves are weaker versions of these homologies. Allen and Swenberg \citep{AS20} presented empirical evidence suggesting that Jones Polynomial can detect causality in $2+1$-dimensional globally hyperbolic spacetime but the Alexander Conway Polynomial failed to distinguish between the connected sum of two Hopf links — which corresponds to the link of skies of two causally unrelated events — and the Allen-Swenberg links which are likely skies of non-causally related events, implying the Alexander-Conway polynomial's inability to detect causality. \\

A \emph{quandle} is an algebraic structure that satisfies axioms analogous to Reidemeister moves, making it an invariant of knots and links. In \citep{Joyce82}, Joyce associated a quandle $Q$ over a finite set with each link $L \subset \mathbb{R}^3$ and illustrated that this \emph{knot quandle} $(Q(L))$ (referred to as a fundamental quandle) determines the knot type up to orientation-preserving homeomorphism of $(L, \mathbb{R}^3)$. Matveev \citep{Mat84} employed a similar structure as a knot invariant and called it distributive groupoids. A useful method of obtaining a computable knot invariant is computing the cardinality of the homomorphism set of the fundamental quandle $Q(L)$ into another fixed quandle $T$. In \citep{Nels08a} and \citep{Nels08b}, Nelson proposed another invariant termed \emph{enhanced quandle counting polynomial}, which extracts more information from the homomorphism set by counting the cardinalities of image subquandles for each $f \in Hom(Q(L),T)$. The Symplectic Quandle \citep{Nels08b} usually has connected non-trivial subquandles, making it suitable for distinguishing links using enhanced quandle counting polynomial. \\

Recent studies have tried to examine if augmenting the Alexander-Conway polynomial with quandle invariants can detect causality in (2+1)-dimensional globally hyperbolic spacetime. Leventhal \citep{Lev23} found that enhancing the Alexander-Conway polynomial with the affine Alexander Quandle isn't adequate for causality detection. This limitation similarly applies to Takasaki quandle, given that it is a special Alexander Quandle with $t=-1$. This paper probes whether the Symplectic Quandle combined with the Alexander-Conway polynomial can detect causality. We find that it possibly can detect causality using the enhanced quandle counting polynomial invariants. \\

The paper is organized as follows: Section 2 recalls important definitions and theorems pertaining to causality in spacetime. Section 3 offers a brief overview of quandles as link invariants. Section 4 defines the symplectic quandle and its properties. Section 5 describes our results. 

\section{Causality in Spacetime}
A \emph{spacetime} $X$ is a time-oriented Laurentz Manifold with an operation known as the \emph{Laurentz Dot Product}, which is defined in each tangent space as follows:
$$(x_1, x_2, \dots, x_n, t_1) \cdot (y_1, y_2, \dots, y_n, t_2) = x_1y_1 + x_2y_2 + \dots + x_ny_n - t_1t_2$$

Two points (events) $x,y \in X$ are causally related if there is a curve $\gamma$ connecting the two points such that $\gamma^{'}(t) \cdot \gamma^{'}(t) \leq 0$, meaning the dot product of the velocity vector with itself must be timelike or null. In other words, one can travel from $x$ to $y$ without exceeding the speed of light. \\

A \emph{Cauchy surface} $\Sigma$ is defined as a subset of the spacetime which is crossed exactly once by every possible maximal causal path as defined above. A spacetime is called \textit{globally hyperbolic} if it has a Cauchy surface. 

\begin{theorem}[Geroch \citep{Ger70}, Bernal and Sanchez \citep{BS03}]
Globally hyperbolic spacetimes are continuous and differentiable equivalent to $\Sigma \times \mathbb{R}^n$ with the $\mathbb{R}$ coordinate being a timelike function and each $\Sigma \times t$ being a Cauchy surface. 
\end{theorem}

Geroch \citep{Ger70} established homeomorphism from the Cauchy surface to the spacetime. Bernal and Sanchez \citep{BS03} demonstrated that the Cauchy Surface can assumed to be smooth and spacelike. 

\begin{theorem}[Bernal and Sanchez \citep{BS07}]
    To maintain global hyperbolicity, a spacetime must satisfy the following two conditions:
    \begin{itemize}
        \item Absence of time travel i.e. absence of closed causal loops
        \item for all $x, y \in X$, $J^+(x) \cap J^-(y)$ is compact where $J^{+}(x)$ and $J^-(y)$ are causal future and causal past, respectively.
    \end{itemize}
\end{theorem}
The second condition is known as the \textit{absence of naked singularities}.\\

By using the space of the spherical contangent bundle $ST*\Sigma$ of a Cauchy surface $\Sigma$, which for $\Sigma = \mathbb{R}^2$ appears like a solid torus, we can find the space of future-directed light rays, termed null geodiscs, $N$ in our globally hyperbolic spacetime.  We get a sphere of light rays passing through a point $x$, called the \emph{sky} of $x$. Low's conjecture establishes the relationship between the causality and the link.

\begin{theorem}[Low's Conjecture \citep{Low88}]
Assume that the universal cover of a smooth spacelike Cauchy surface of a globally hyperbolic (2+1)-dimensional spacetime $(X, g)$ is diffeomorphic to $\mathbb{R}^2$. Then the skies of causally related points in X are topologically linked.
\end{theorem}

Essentially, the conjecture posits that the study of Causality can be associated with the study of links in the solid torus. It says that two events $x, y$ are causally related if and only if the resulting link from the skies $S_x \sqcup S_y$ is non-trivial. This theorem was proved by Chernov and Nemirovski \citep{CN10}, where they showed that the relation holds if the Cauchy surface $\Sigma$ is not homeomorphic to $S^2$ or $\mathbb{R}P^2$. 

\begin{theorem}[Legendrian Low's Conjecture]
Two points $x$ and $y$ in a $(3 + 1)$ dimension globally hyperbolic space with Cauchy surface $\mathbb{R}^3$ are causally related if their skies are Legendrian linked.
\end{theorem}

The conjecture was postulated by Natàrio and Tod \citep{Natario04} and was later proved by Chernov and Nemirovski \citep{CN10}. 

\section{Quandles as Invariants of Knots and Links}

\begin{definition}
A quandle is a non empty set $X$ with a binary operation $\triangleright$ that satisfies the following axioms:
    \begin{itemize}
        \item $x \triangleright x = x$ for all $x \in X$
        \item Right multiplication by $x$ is bijective for all $x$. Specifically, for elements $x, y \in X$, there exists an element $z \in X$ such that $x = y \triangleright z$
        \item $(x \triangleright y) \triangleright z = (x \triangleright z) \triangleright (y \triangleright z)$ for all $x,y,z \in X$
    \end{itemize}
    
\end{definition}
\noindent The first axiom indicates that quandles are idempotent. The second axiom state the existence of right inverse, $\triangleright^{-1}$, such that $(x \triangleright y) \triangleright^{-1} y = x$. The third axiom highlights that quandles are right self-distributive. Moreover, quandles are generally non-commutative and non-associative, meaning $x \triangleright y \neq y \triangleright x$ and $(x \triangleright y) \triangleright z \neq x \triangleright (y \triangleright z)$. It is also important to note that quandle axioms cover the Reidmeister moves, making it a good algebraic invariant to study knots and links (see \citep{Nel04}).\\

Joyce, in \citep{Joyce82}, delineated the relationship between knots and quandles, where arcs are denoted by quandle elements and the relationship between arcs at each crossing is represented by the quandle operation, as shown in figure \ref{fig:crossing} below:
\begin{figure}[h]
    \centering
    \includegraphics[width = 10cm]{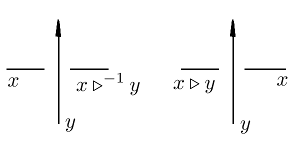}
    \caption{The Quandle Crossing Relations}
    \label{fig:crossing}
\end{figure}
\FloatBarrier
The left diagram in Figure \ref{fig:crossing} illustrates a negative crossing, where the arc $x$ passes underneath $y$ from the left resulting in arc $x \triangleright^{-1} y$. Conversely, the right diagram depicts a positive crossing, where the arc $x$ passes underneath $y$ from right to left, resulting in arc $x \triangleright y$. We will use the positive crossing to represent the relationship between knots and quandles. \\

The Fundamental Quandle, $Q(L)$, of an oriented knot or a link ($L$) is the quandle generated by arcs $L$ with relationship defined at each crossing as shown in Figure \ref{fig:crossing}. 

\begin{definition}
Given an oriented link $L$ with arc-set $A$, the fundamental quandle $Q(L)$ of $L$ is a quandle on $A$ with the quandle operation defined by crossing relations. 
\end{definition}

Joyce \citep{Joyce82} demonstrated that the fundamental quandle is a \emph{complete invariant} of knots up to an ambient homeomorphism. 

\begin{definition}
Let $X_1$ and $X_2$ be two quandles with operation $\triangleright$. A function $\varphi : X_1 \mapsto X_2$ is a quandle homomorphism if $\varphi(x \triangleright y) = \varphi(x) \triangleright \varphi(y)$ for all $x, y \in X_1$.
\end{definition}

\begin{theorem}[Joyce \citep{Joyce82}] 
There exists a homeomorphism $f: S^3 \rightarrow S^3$ taking an oriented knot $K$ to another oriented knot $K^{'}$ if and only if the fundamental quandles $Q(K)$ and $Q(K^{'})$ are isomorphic. 
\end{theorem}

Since differentiating two links by just using the fundamental quandle is difficult, quandle invariants often employ homomorphism of the fundamental quandle into another finite quandle to formulate invariants. 
If we assign each arc of a link (L) to an element in a finite quandle $T$, and then apply the quandle relations of the fundamental quandle $Q(L)$ at each crossing to $T$, we obtain a system of equations. Solving this system of equations gives the homomorphism set $Hom(Q(L),T)$. These homomorphisms are referred to as the coloring of the link by elements of $T$. The cardinality of the set $|Hom(Q(L),T)|$ is a link invariant. 

\begin{definition}
The \emph{coloring} of a knot or link $L$ is defined as the assignment of values from a quandle $T$ to the arcs of $L$ such that the values respect the $\triangleright$ operation at each crossing. 
\end{definition}

\begin{definition}
For an oriented link $L$ with its fundamental quandle $Q(L)$ and a finite quandle $T$ (termed the coloring quandle), the set of quandle homomorphism from $Q(L)$ to $T$, denoted as $Hom(Q(L), T)$, is called the coloring space. The cardinality of the coloring space, $|Hom(Q(L),T)| = \phi_T(L)$ is called the \emph{quandle counting invariant}. 
\end{definition}

\noindent Quandle counting invariants sometimes may not be sufficient in distinguishing links as they omit information contained in the homomorphism set. An enhancement to this invariant is quandle 2-cocycle invariants $\phi_{\chi}(L,X)$ \citep{Carter05}, which count the homomorphism weighted by a cocycle, extracting additional information from the homomorphism set. Nelson \citep{Nels08a} proposed a dual-variable \emph{subquandle polynomial invariant}, termed $\phi_{qp}(L)$, which uses the image of each homomorphism. Nelson in \citep{Nels08b} defined a specialized version of the subquandle polynomial invariant named \emph{enhanced quandle counting invariant}. This counts the cardinality of the image of each homomorphism $f$, thus obtaining a set with multiplicities of integers, which is transformed into a polynomial as defined below.

\begin{definition}
The enhanced quandle counting invariant of a link $L$ with respect to a finite quandle $T$ is given by:
$$\Phi_E(L, T) = \sum_{f \in Hom(Q(L), T)} q^{|Im(f)|}$$
\end{definition}

Examples of frequently utilized coloring quandles include:
\begin{itemize}
\item[(a)] \emph{Takasaki kei:} $T = \mathbb{Z}_n$ with the quandle operation $(x \triangleright y) = 2 y - x \pmod{n}$.
\item[(b)] \emph{Alexander quandles:} $T$ is a module over the ring of Laurent polynomial $\mathbb{Z}[t, t^{-1}]$ with the quandle operation as $(x \triangleright y) = t x + (1-t) y$. When $t = -1$, this is equivalent to the Takasaki quandle if $t = -1$. 
\end{itemize}

We use symplectic quandles (described in the next section) for coloring. 

\section{Symplectic Quandle}
A symplectic quandle is a fixed quandle that employs bilinear forms over R-modules to define the quandle operation (see \citep{Nels08b} and \citep{Yett03}). 
\begin{definition}
    Let $M$ be a finite vector space defined over a commutative ring $R$ and let $\langle,\rangle : M \times M \mapsto R$ be an anti-symmetric bilinear form such that $\langle x, x \rangle = 0$ for all $x \in M$. Then M is a quandle with the operation
    $$x \triangleright y = x + \langle x, y \rangle y$$
    The corresponding right inverse operation is:
    $$x \triangleright^{-1} y = x - \langle x, y \rangle y$$
\end{definition}
It can easily be shown that the above structure satisfies all three axioms of quandles (see \citep{Yett03}). Navas and Nelson in \citep{Nels08b} established that if a symplectic quandle $M$ is defined over a finite field and has a non-degenerate form that is $\langle x, y \rangle = 0$ for all $y \in M$ implies $x = 0 \in M$, then $M$ includes a connected sub-quandle. 
\begin{definition}
A quandle, $M$, is considered connected if every element $a \in M$ can be obtained from every other element $b \in M$ through a sequence of $\triangleright$ and $\triangleright^{-1}$ operations. 
\end{definition}

\begin{theorem}[Navas and Nelson \citep{Nels08b}]
If a symplectic quandle $M$ is defined over a finite field and $\langle,\rangle$ is non-degenerate, then the subquandle $M \setminus \{0\}$ is a connected quandle.
\end{theorem}

The presence of a connected subquandle is noteworthy because a fundamental quandle of a link is invariably connected, and hence its homomorphism to a symplectic quandle yields non-trivial solutions with a number of non-trivial subquandles. This makes the symplectic quandle apt for the enhanced quandle counting invariant. \\

It is important to recognize that the non-degenerate anti-symmetric bilinear form requires an even dimension. For our computation in the next section, we use the 2-dimension symplectic quandles defined over finite field $\mathbb{Z}_p$ where $p \in \{2,3,5\}$. 

\section{Detecting Causality using Symplectic Quandle}

When a Cauchy surface is homeomorphic to $\mathbb{R}^2$, two events that are causally unrelated have a sky link that is isotopic to the pair of parallels on the solid torus  $S^1 \times \mathbb{R}^2$. When we study this link in $\mathbb{R}^3$, we get the connected sum of Hopf Links, denoted as $H$. A third knot component is necessary to identify the solid torus with $\mathbb{R}^3$ after we delete this third component. \\

\begin{comment}
\textcolor{red}{In a Cauchy surface homeomorphic to $\mathbb{R}^2$, when two events are causally unrelated, their skies will be circles present on the solid torus $S^1 \times \mathbb{R}^2$. By Low's conjecture these two circles represent a trivial link, so they will be disjoint. When we study this link in $\mathbb{R}^3$, we get the connected sum of Hopf Links (denoted as $H$).} \\
\end{comment}

Allen and Swenberg in \citep{AS20} constructed an infinite series of links (Allen-Swenberg links), which cannot be differentiated from $H$ using  the Alexander-Conway polynomial. They suggested that to ascertain an invariant's capability to detect causality, it must be able to distinguish these links from H. We denote the Allen-Swenberg links as $L_n$ where $n \in \mathbb{Z}_{+}$. The first link in the series, $L_1$, is depicted in figure \ref{fig:aslink}. Our objective is to assess if combining the Alexander Conway polynomial with a Symplectic quandle enables differentiation of Allen-Swenberg links\footnote{Allen and Swenberg \citep{AS20} identified the infinite family of links as $C(1*T_0(n))$; $C(0)$ represents connected sum of two Hopf links $H$.} from $H$. 

\begin{figure}[h]
    \centering
    \includegraphics[width=7cm]{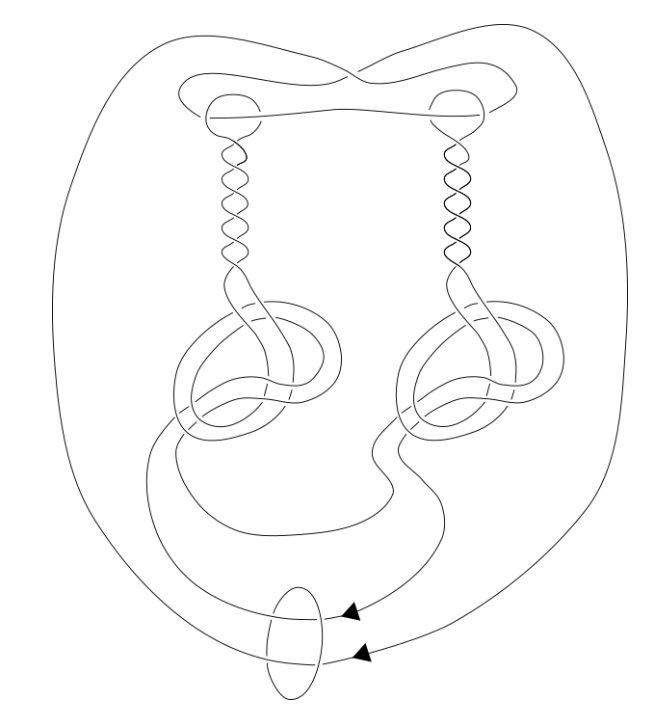}
    \caption{First link, $L_1$, of Allen-Swenberg links}
    \label{fig:aslink}
\end{figure}
\FloatBarrier
Let us define a symplectic quandle $T = (\mathbb{Z}_5)^2$ over the field $\mathbb{Z}_5$ with its bilinear form characterized by a $2 \times 2$ anti-symmetric invertible matrix $A$.
 $$A \equiv \begin{pmatrix}. 
    0 & 1\\
    -1 & 0 \\
\end{pmatrix} $$
For any pair of vectors $x, y \in T$ , $x \triangleright y  \equiv x + (xAy^{T}) \cdot y \pmod{5}$. We denote the fundamental quandle of any link $L$ as $Q(L)$ and the set of homomorphism from $Q(L)$ to $T$ as $Hom(Q(L),T)$. 

\subsection{Homomorphism of the fundamental quandle of H to T}
Figure \ref{fig:labhopf} shows the labeled connected sum of Hopf links ($H$).   
\begin{figure}[ht]
    \centering
    \includegraphics[width = 10cm]{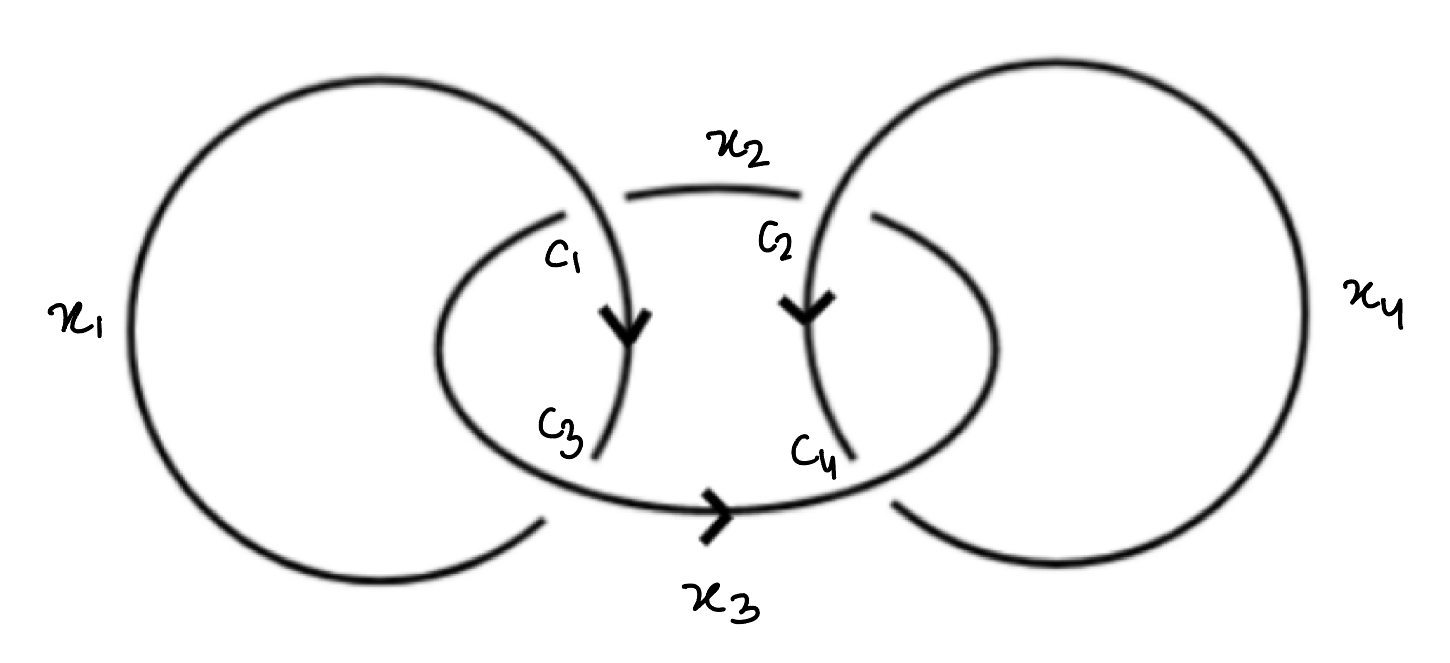}
    \caption{Labelled Connected Sum of Hopf Links (H)}
    \label{fig:labhopf}
\end{figure}
\FloatBarrier
Each crossing in figure \ref{fig:labhopf} represents the quandle operation under the fundamental quandle. We map each arc to a vector $x_i \in T$ and apply The same quandle operation to $T$ to find homomorphism of $Q(H)$ into $T$. The subsequent table presents the fundamental quandle and the associated equation derived for the symplectic quandle $T$ at each crossing.
\begin{center}
\resizebox{14cm}{!}{
    \begin{tabular}{|c|c|c|}
    \hline
    Crossing &  Fundamental Quandle $Q(H)$ & Symplectic Quandle $T$\\
    \hline 
    $c_1$ & $x_2 = x_3 \triangleright x_1$& $x_2 = x_3 + (x_3Ax_1^{T}) x_1$ \\
    \hline
    $c_2$ & $x_3 = x_2 \triangleright x_4$& $x_3 = x_2 + (x_2Ax_4^{T}) x_4$ \\
    \hline 
    $c_3$ & $x_1 = x_1 \triangleright x_3$ &  $x_1 = x_1 + (x_1Ax_3^{T}) x_3$\\
    \hline
    $c_4$ & $x_4 = x_4 \triangleright x_3$ & $x_4 = x_4 + (x_4Ax_1^{T}) x_3$\\
    \hline
\end{tabular}
}
\end{center}
\FloatBarrier
The solutions to the system of equations shown under $T$ give us the homomorphism set $Hom(Q(H), T)$. 
Given that each arc variable is a vector in the vector space $(Z_5)^2$, this system has $4 \times 2 = 8$ equations with $8$ unknown. Because the symplectic quandle leads to non-linear equations, we cannot solve the system by using Gaussian elimination or row reduction. We use Wolfram Mathematica \citep{Wol23} to solve this system. We obtain a total of $1225$ solutions, and the relationship $x_2 = x_3$ holds in all solutions. Each solution signifies a coloring scheme of arcs, with each vector indicating a color. If we cluster the arcs by their colors, the homomorphism set displays the following groupings (arcs inside the braces share the same color):\\

\noindent Monochromatic coloring: \{1,2,3,4\}\\
Dual coloring has 3 configurations: \{\{1, 2, 3\}, \{4\}\}; \{\{1, 4\}, \{2, 3\}\}; \{\{1\}, \{2, 3, 4\}\}\\
Triple coloring has 1 configuration: \{\{1\}, \{2, 3\}, \{4\}\}\\

\noindent Note: Each configuration has multiple solutions to represent different possible color combinations. For instance, 25 solutions exist for monochromatic coloring, one for each color. \\

$1225$ solutions consist of 25 solutions that use a single color i.e. $|Im(f)| = 1$, 360 solutions with two distinct colors, i.e., $|Im(f)| = 2$, and 840 solutions with three distinct colors, i.e., $|Im(f)| = 3$. There are no solutions that use four distinct colors because the arcs $x_2$ and $x_3$ always have the same color. Consequently, the cardinality of the homomorphism set, and the corresponding enhanced quandle counting polynomial are given by: 
$$|Hom(Q(H),T)| = 1225 \text{ and } \phi_E(H,T) = 25 q + 360 q^2 + 840 q^3$$. 
Notice that \emph{for every $f:(x_1,x_2,x_3,x_4) \in Hom(Q(H),T)$, $x_2 = x_3$ which implies that maximum number of colors in any $f$ is $3$ and hence the degree of polynomial $\phi_E(H,T)$ is $3$}.

\subsection{Homomorphism of the fundamental quandle of $L_1$ to $T$}
Now, we follow the identical procedure for the first Allen-Swenberg link (denoted as $L_1$), which has been labeled as shown in figure \ref{fig:labelaslink}. The labeling mirrors that shown in Figure 7 of Leventhal \citep{Lev23} for easier comparisons. 
\begin{figure}[h]
    \centering
    \includegraphics[width=14cm]{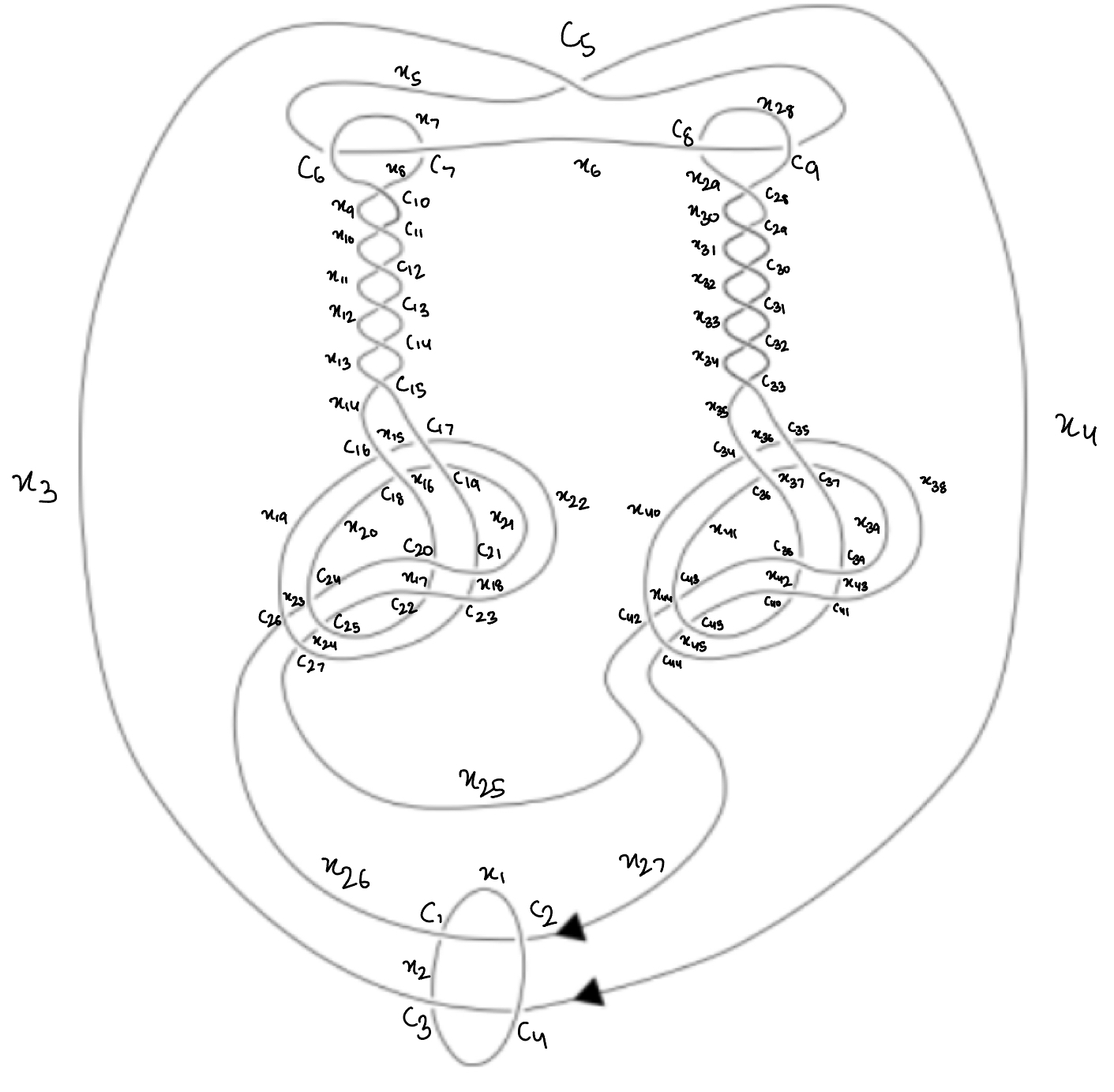}
    \caption{Labelled Allen-Swenberg Link $(L_1)$}
    \label{fig:labelaslink}
\end{figure}
\FloatBarrier
The subsequent table outlines the quandle relations under fundamental quandle $Q(L_1)$. For the sake of brevity, we've omitted the associated symplectic quandle operation for each crossing. However, it is given by $x \triangleright y \equiv x + (x A y^{T})y$ where $x, y$ are corresponding arcs under the fundamental quandle.
\begin{center}
\resizebox{14cm}{!}{
 \begin{tabular}{|c|c|}
    \hline
    Crossing & $Q(L_1)$ \\
    \hline
    $c_1$ & $x_2 = x_1 \triangleright x_{26}$ \\
    \hline
    $c_2$ & $x_{26} = x_{27} \triangleright x_1$ \\
    \hline
    $c_3$ & $x_1 = x_2 \triangleright x_3$ \\
    \hline
    $c_4$ & $x_3 = x_4 \triangleright x_1$ \\
    \hline
    $c_5$ & $x_4 = x_5 \triangleright x_3$ \\
    \hline
    $c_6$ & $x_5 = x_6 \triangleright x_7$ \\
    \hline
    $c_7$ & $x_8 = x_7 \triangleright x_6$ \\
    \hline
    $c_8$ & $x_{29} = x_{28} \triangleright x_6$ \\
    \hline
    $c_9$ & $x_3 = x_6 \triangleright x_{28}$ \\
    \hline
    $c_{10}$ & $x_9 = x_8 \triangleright x_7$ \\
    \hline
    $c_{11}$& $x_7 = x_{10} \triangleright x_9$ \\
    \hline
    $c_{12}$& $x_{11} = x_9 \triangleright x_{10}$ \\
    \hline
    $c_{13}$& $x_{10} = x_{11} \triangleright x_{12}$ \\
    \hline
    $c_{14}$& $x_{13} = x_2 \triangleright x_4$ \\
    \hline
    $c_{15}$& $x_{12} = x_{14} \triangleright x_{13}$ \\
    \hline
    $c_{16}$& $x_{19} = x_{15} \triangleright x_{14}$ \\
    \hline
    $c_{17}$& $x_{22} = x_{15} \triangleright x_{13}$ \\
    \hline
    $c_{18}$& $x_{20} = x_{16} \triangleright x_{14}$ \\
    \hline
    $c_{19}$& $x_{21} = x_{16} \triangleright x_{13}$ \\
    \hline
    $c_{20}$& $x_{14} = x_{17} \triangleright x_{21}$ \\
    \hline
    $c_{21}$& $x_{13} = x_{18} \triangleright x_{21}$ \\
    \hline
    $c_{22}$& $x_{20} = x_{17} \triangleright x_{22}$ \\
    \hline
   $c_{23}$& $x_{19} = x_{18} \triangleright x_{22}$ \\
    \hline
\end{tabular}

 \begin{tabular}{|c|c|}
    \hline
    Crossing &  Quandle Relation\\
    \hline
    $c_{24}$& $x_{21} = x_{23} \triangleright x_{20}$ \\
    \hline
    $c_{25}$& $x_{22} = x_{24} \triangleright x_{20}$ \\
    \hline
    $c_{26}$& $x_{26} = x_{23} \triangleright x_{19}$ \\
    \hline
    $c_{27}$& $x_{25} = x_{24} \triangleright x_{19}$ \\
    \hline
    $c_{28}$& $x_{30} = x_{28} \triangleright x_{29}$ \\
    \hline
    $c_{29}$& $x_{29} = x_{31} \triangleright x_{30}$ \\
    \hline
    $c_{30}$& $x_{32} = x_{30} \triangleright x_{31}$ \\
    \hline
    $c_{31}$& $x_{31} = x_{33} \triangleright x_{32}$ \\
    \hline
    $c_{32}$& $x_{34} = x_{32} \triangleright x_{33}$ \\
    \hline
    $c_{33}$& $x_{33} = x_{35} \triangleright x_{34}$ \\
    \hline
    $c_{34}$& $x_{40} = x_{36} \triangleright x_{35}$ \\
    \hline
    $c_{35}$& $x_{38} = x_{36} \triangleright x_{34}$ \\
    \hline
    $c_{36}$& $x_{41} = x_{37} \triangleright x_{35}$ \\
    \hline
    $c_{37}$& $x_{39} = x_{37} \triangleright x_{34}$ \\
    \hline
    $c_{38}$& $x_{35} = x_{42} \triangleright x_{39}$ \\
    \hline
    $c_{39}$& $x_{34} = x_{43} \triangleright x_{39}$ \\
    \hline
    $c_{40}$& $x_{41} = x_{42} \triangleright x_{38}$ \\
    \hline
    $c_{41}$& $x_{40} = x_{43} \triangleright x_{38}$ \\
    \hline
    $c_{42}$& $x_{25} = x_{44} \triangleright x_{40}$ \\
    \hline
    $c_{43}$& $x_{39} = x_{44} \triangleright x_{41}$ \\
    \hline
    $c_{44}$& $x_{27} = x_{45} \triangleright x_{40}$ \\
    \hline
    $c_{45}$& $x_{38} = x_{45} \triangleright x_{41}$ \\
    \hline
\end{tabular}
}
\end{center}
The system of equations presented above comprises $45 \times 2 = 90$ equations with $90$ unknowns, given that $T = (\mathbb{Z}_5)^2$. The size of the solution space of such a system would be $5^{90}$, which makes it impossible for any brute-force algorithm to compute within a feasible timeframe. Moreover, these equations are non-linear, rendering linear solution methods like Gaussian elimination inapplicable. Indeed, Wolfram Mathematica \citep{Wol23} was unable to solve this system. \\ 

Nonetheless, we could use a divide-and-conquer type algorithm to solve this system efficiently. Such an algorithm (e.g. merge sort) recursively breaks down a problem into smaller sub-problems, until these sub-problems are small enough to be solved directly. The solutions to the sub-problems are then combined to give a solution to the original problem.\\

In our approach, we divide the link into multiple parts, each encompassing 2-3 crossings. We solve the equations of these smaller parts, and then we combine the solution sets of these parts using the join operation in Mathematica. The join operation combines the solution sets of two adjacent pieces in a manner that the values of common variables are equal. It is important to note that atleast one common variable exists when parts are adjacent. The join operation is significantly less computationally intensive for three reasons: a) it operates on the reduced solution set rather than the entirety of potential solutions, b) it matches based on a subset of shared variables, and c) it is a less complex operation as it requires combining two matching sets instead of evaluating full system of equations.\\

A part with $3$ crossings will encompass at most $5$ arc variables, hence the solution space of such part is $5^{10}$ which is significantly smaller than $5^{90}$. Since a join operation is less computational intensive, we can efficiently combine the solutions of smaller parts into the final solution set. Indeed, using Mathematica \citep{Wol23} the solution set was promptly generated in few minutes. More importantly, this algorithm can be seamlessly up-scaled to $L_2$ with 85 crossings, as the computation time increases proportionately with the number of crossings and not exponentially. \\

Upon solving the system of equations for $L_1$, we obtains $1225$ solutions. These include 25 solutions where all arcs have a single color ($|Im(f)| = 1$), 360 solutions with two distinct colors ($|Im(f)| = 2$), 360 solutions with three distinct colors($|Im(f)| = 3$), 360 solutions with 21 distinct colors($|Im(f)| = 21$), and 120 solutions with 22 distinct colors ($|Im(f)| = 22$). \\

We can get more specific information about the homomorphism set, if we group the arcs having identical values (or colors) together as shown in the table below. Note that a specific coloring scheme can have multiple color combinations within the homomorphism set. For instance, a monochromatic coloring scheme can accommodate 25 colors (in $\mathbb{Z}_5$).
\begin{center}
\small{Coloring scheme with 2 colors (3 configurations):}\\
\small{\{\{1, 2\}, \{3,4,5,6,7-45\}\},\\
\{\{1, 2, 7-45\}, \{3, 4, 5, 6\}\},\\ 
\{\{1, 2, 3, 4, 5, 6\}, \{7-45\}\}} \\
Coloring scheme using 3 colors (1 configuration):\\
\small{\{\{1, 2\}, \{3, 4, 5, 6\}, \{7-45\}\} }\\
Coloring scheme using 21 colors (3 configurations):\\
\small{\{\{1, 2\}, \{3, 4, 5, 34, 42\}, \{6, 11, 19\}, \{7, 24, 26, 27, 28\}, \{8, 25, 29\}, \{9, 22\}, \{10, 15, 21\}, \{12, 20\}, \{13\}, \{14\}, \{16, 43\}, \{17, 33, 41\}, \{18, 37\}, \{23\}, \{30, 38\}, \{31, 39\}, \{32,40\}, \{35\}, \{36\}, \{44\}, \{45\}\},\\
\{\{1, 2\}, \{3, 4, 5, 12, 20\},\{6,35\},\{7, 24, 26, 27, 28\}, \{8, 25, 29\}, \{9, 22\}, \{10, 15, 21\}, \{11, 19\}, \{13\}, \{14\}, \{16, 43\}, \{17, 33, 41\}, \{18, 37\}, \{23\}, \{30, 38\}, \{31, 39\}, \{32, 40\}, \{34, 42\}, \{36\}, \{44\}, \{45\}\}, \\
\{\{1, 2\}, \{3, 4, 5, 15\}, \{6, 9, 22\}, \{7, 26, 27, 28\}, \{8, 25, 29, 45\}, \{10, 21\}, \{11, 19\}, \{12, 20, 42\}, \{13, 17\}, \{14\}, \{16, 43\}, \{18, 37\}, \{23\}, \{24\}, \{30, 38\}, \{31, 36, 39\}, \{32, 40\}, \{33,41\}, \{34\}, \{35\}, \{44\}\}
   }\\
   Coloring scheme using 22 colors (1 configuration):\\
\small{\{\{1, 2\}, \{3, 4, 5, 10, 21\}, \{6\}, \{7, 26, 27, 28\}, \{8, 25, 29, 45\}, \{9, 22\}, \{11, 19\}, \{12, 20, 42\}, \{13, 17\}, \{14\}, \{15\}, \{16,43\}, \{18, 37\}, \{23\}, \{24\}, \{30, 38\}, \{31, 36, 39\}, \{32, 40\}, \{33, 41\}, \{34\}, \{35\}, \{44\}\}
}\\
%Note: The arcs within the innermost curly braces are of the same colors. 
\end{center}
Observe that the coloring space is characterized as: \emph{for every $f:(x_1,\cdots, x_{45})  \in Hom(Q(L_1),T)$, $x_1 = x_2$, $x_3 = x_4 = x_5$, and $x_{26} = x_{27}$}. \\

Based on the above result, the number of homomorphisms from $Q(L_1)$ to $T$, and the enhanced quandle counting polynomial, are given by:
$$|Hom(Q(L_1),T)| = 1225 \text{ and } \phi_E(L_1,T) = 25 q + 360 q^2 + 360 q^3 + 360 q^{21} + 120 q^{22}$$
It is evident that the invariant measuring the cardinality of the homomorphism set does not differentiate between the two links $H$ and $L_1$ because $|Hom(Q(L_1),T)| = |Hom(Q(H),T)| = 1225$. However, an invariant that uses enhanced quandle counting polynomials can distinguish these two links because $Hom(Q(L_1),T)$ set includes coloring schemes with $21$ and $22$ distinct colors, whereas $Hom(Q(H),T)$ lacks such coloring schemes. Therefore, we can conclude that the symplectic quandle is able to distinguish link $L_1$ from $H$. We corroborated our findings in $\mathbb{Z}_2$ and $\mathbb{Z}_3$, as illustrated below, and observed similar results. 
\begin{center}
\resizebox{18cm}{!}{
    \begin{tabular}{|c|c|c|c|}
    \hline
    Vector Field &  $\phi_E(H,T)$ & $\phi_E(L_1,T)$ & Distinguishes Links?\\
    \hline 
    $Z_2$ & $4 q + 18 q^2 + 6 q^3$ & $4 q + 18 q^2 + 6 q^4$ & Yes \\
    \hline
    $Z_3$ & $9 q + 72 q^2 + 72 q^3$ & $9 q + 72 q^2 + 24 q^3 + 48 q^9$ & Yes \\
    \hline 
    $Z_5$ & $25 q + 360 q^2 + 840 q^3 $ &  $25 q + 360 q^2 + 360 q^3 + 360 q^{21} + 120 q^{22}$ & Yes \\
    \hline
\end{tabular}
}
\end{center}
\FloatBarrier

Since the symplectic quandles have a maximum of 4 colors in $(\mathbb{Z}_2)^2$ and 9 colors in $(\mathbb{Z}_9)^2$, the enhanced counting polynomial can have a maximum degree of 4 and 9, respectively. Therefore, the coloring schemes with $21$ and $22$ distinct colors do not exist in these cases. However, the two links remain distinguishable since the set $Hom(Q(H),T)$ does not have a coloring scheme with 4 or 9 distinct colors. \\

We confirmed that our result is invariant to any non-null anti-symmetric matrix in $(Z_5)^2$. If we use a null matrix, $T$ becomes a trivial quandle and we get the same result as described in \citep{Lev23} (Alexander Quandle with $t =1$) that the quandle cannot distinguish $H$ from $L_1$. 

\subsection{Homomorphism of the fundamental quandle of $L_2$ to $T$}
Now, let's investigate if the symplectic quandle $T$ can differentiate the second link (denoted as $L_2$) in the Allen-Swenberg series from both the connected sum of two Hopf links (H) and the first Allen-Swenberg link ($L_1$). Figure \ref{fig:as2} shows the labeled $L_2$.
\begin{figure}[ht]
    \centering
    \includegraphics[width = 14.5cm]{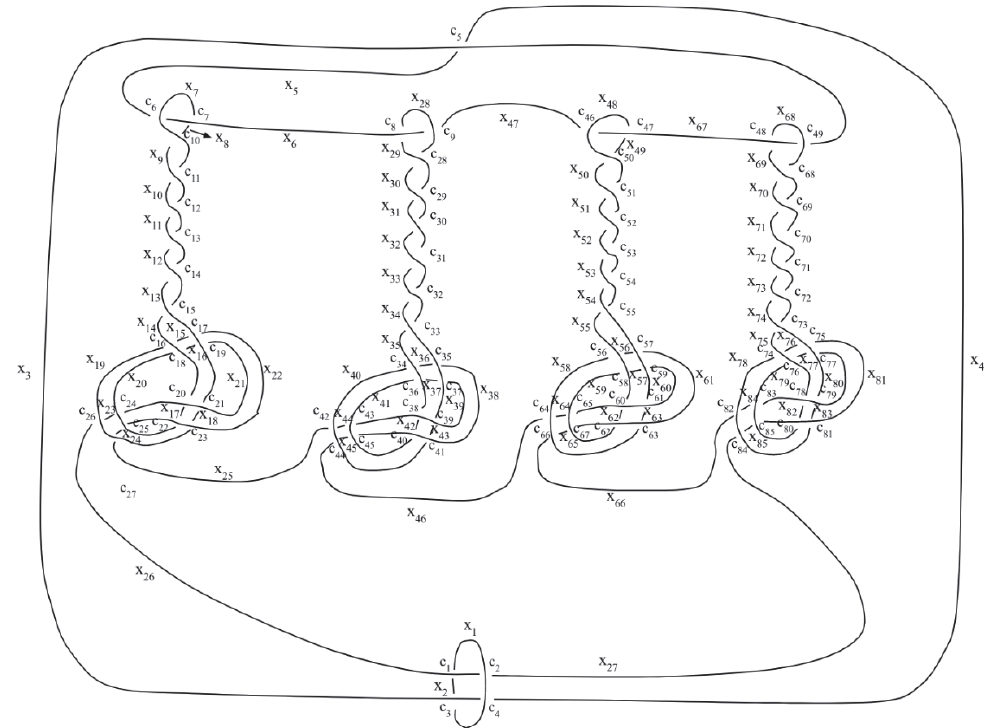}
    \caption{Labeled Allen Swenberg Second Link ($L_2$)}
    \label{fig:as2}
\end{figure}
\FloatBarrier

The table below displays the relationship among the elements of the fundamental quandle $Q(L_2)$ at each crossing. These relationships for the crossing from $c_1$ to $c_{45}$ are identical to those in $Q(L_1)$, with the exception of $c_9$ and $c_{44}$ (which are highlighted).

\resizebox{17cm}{!}{
 \begin{tabular}{|c|c|}
    \hline
    Crossing & $Q(L_2)$ \\
    \hline
    $c_1$ & $x_2 = x_1 \triangleright x_{26}$ \\
    \hline
    $c_2$ & $x_{26} = x_{27} \triangleright x_1$ \\
    \hline
    $c_3$ & $x_1 = x_2 \triangleright x_3$ \\
    \hline
    $c_4$ & $x_3 = x_4 \triangleright x_1$ \\
    \hline
    $c_5$ & $x_4 = x_5 \triangleright x_3$ \\
    \hline
    $c_6$ & $x_5 = x_6 \triangleright x_7$ \\
    \hline
    $c_7$ & $x_8 = x_7 \triangleright x_6$ \\
    \hline
    $c_8$ & $x_{29} = x_{28} \triangleright x_6$ \\
    \hline
    \textcolor{blue}{$c_9$} & \textcolor{blue}{$x_{47} = x_6 \triangleright x_{28}$} \\
    \hline
    $c_{10}$ & $x_9 = x_8 \triangleright x_7$ \\
    \hline
    $c_{11}$& $x_7 = x_{10} \triangleright x_9$ \\
    \hline
    $c_{12}$& $x_{11} = x_9 \triangleright x_{10}$ \\
    \hline
    $c_{13}$& $x_{10} = x_{11} \triangleright x_{12}$ \\
    \hline
    $c_{14}$& $x_{13} = x_2 \triangleright x_4$ \\
    \hline
    $c_{15}$& $x_{12} = x_{14} \triangleright x_{13}$ \\
    \hline
    $c_{16}$& $x_{19} = x_{15} \triangleright x_{14}$ \\
    \hline
    $c_{17}$& $x_{22} = x_{15} \triangleright x_{13}$ \\
    \hline
    $c_{18}$& $x_{20} = x_{16} \triangleright x_{14}$ \\
    \hline
    $c_{19}$& $x_{21} = x_{16} \triangleright x_{13}$ \\
    \hline
    $c_{20}$& $x_{14} = x_{17} \triangleright x_{21}$ \\
    \hline
    $c_{21}$& $x_{13} = x_{18} \triangleright x_{21}$ \\
    \hline
    $c_{22}$& $x_{20} = x_{17} \triangleright x_{22}$ \\
    \hline
   $c_{23}$& $x_{19} = x_{18} \triangleright x_{22}$ \\
   \hline
     $c_{24}$& $x_{21} = x_{23} \triangleright x_{20}$ \\
    \hline
    $c_{25}$& $x_{22} = x_{24} \triangleright x_{20}$ \\
    \hline
    $c_{26}$& $x_{26} = x_{23} \triangleright x_{19}$ \\
    $c_{27}$& $x_{25} = x_{24} \triangleright x_{19}$ \\
    \hline
    $c_{28}$& $x_{30} = x_{28} \triangleright x_{29}$ \\
    \hline
\end{tabular}

 \begin{tabular}{|c|c|}
    \hline
    Crossing & $Q(L_2)$\\
    \hline
    $c_{29}$& $x_{29} = x_{31} \triangleright x_{30}$ \\
    \hline
    $c_{30}$& $x_{32} = x_{30} \triangleright x_{31}$ \\
    \hline
    $c_{31}$& $x_{31} = x_{33} \triangleright x_{32}$ \\
    \hline
    $c_{32}$& $x_{34} = x_{32} \triangleright x_{33}$ \\
    \hline
    $c_{33}$& $x_{33} = x_{35} \triangleright x_{34}$ \\  
    \hline  
    $c_{34}$& $x_{40} = x_{36} \triangleright x_{35}$ \\
    \hline
    $c_{35}$& $x_{38} = x_{36} \triangleright x_{34}$ \\
    \hline
    $c_{36}$& $x_{41} = x_{37} \triangleright x_{35}$ \\
    \hline
    $c_{37}$& $x_{39} = x_{37} \triangleright x_{34}$ \\
    \hline
    $c_{38}$& $x_{35} = x_{42} \triangleright x_{39}$ \\
    \hline
    $c_{39}$& $x_{34} = x_{43} \triangleright x_{39}$ \\
    \hline
    $c_{40}$& $x_{41} = x_{42} \triangleright x_{38}$ \\
    \hline
    $c_{41}$& $x_{40} = x_{43} \triangleright x_{38}$ \\
    \hline
    $c_{42}$& $x_{25} = x_{44} \triangleright x_{40}$ \\
    \hline
    $c_{43}$& $x_{39} = x_{44} \triangleright x_{41}$ \\
    \hline
    \textcolor{blue}{$c_{44}$}& \textcolor{blue}{$x_{46} = x_{45} \triangleright x_{40}$} \\
    \hline
    $c_{45}$& $x_{38} = x_{45} \triangleright x_{41}$ \\
  	\hline
    $c_{46}$& $x_{47} = x_{67} \triangleright x_{48}$ \\
    \hline
    $c_{47}$& $x_{49} = x_{48} \triangleright x_{67}$ \\
    \hline
    $c_{48}$& $x_{69} = x_{68} \triangleright x_{67}$ \\
    \hline
    $c_{49}$& $x_{3} = x_{67} \triangleright x_{68}$ \\
    \hline
    $c_{50}$& $x_{50} = x_{49} \triangleright x_{48}$ \\
    \hline
    $c_{51}$& $x_{48} = x_{51} \triangleright x_{50}$ \\
    \hline
    $c_{52}$& $x_{52} = x_{50} \triangleright x_{51}$ \\
    \hline
    $c_{53}$& $x_{51} = x_{53} \triangleright x_{52}$ \\
    \hline
    $c_{54}$& $x_{54} = x_{52} \triangleright x_{53}$ \\
    \hline
    $c_{55}$& $x_{53} = x_{55} \triangleright x_{54}$ \\
    \hline
    $c_{56}$& $x_{58} = x_{56} \triangleright x_{55}$ \\  
    \hline
\end{tabular}

 \begin{tabular}{|c|c|}
    \hline
    Crossing & $Q(L_2)$ \\
	\hline 
    $c_{57}$& $x_{61} = x_{56} \triangleright x_{54}$ \\
    \hline
    $c_{58}$& $x_{59} = x_{57} \triangleright x_{55}$ \\
    \hline
    $c_{59}$& $x_{60} = x_{57} \triangleright x_{54}$ \\
    \hline
    $c_{60}$& $x_{55} = x_{62} \triangleright x_{60}$ \\
    \hline
    $c_{61}$& $x_{54} = x_{63} \triangleright x_{60}$ \\
    \hline
    $c_{62}$& $x_{59} = x_{62} \triangleright x_{61}$ \\
    \hline
    $c_{63}$& $x_{58} = x_{63} \triangleright x_{61}$ \\
    \hline
    $c_{64}$& $x_{46} = x_{64} \triangleright x_{58}$ \\
    \hline
    $c_{65}$& $x_{60} = x_{64} \triangleright x_{59}$ \\
    \hline
    $c_{66}$& $x_{66} = x_{65} \triangleright x_{58}$ \\
    \hline
    $c_{67}$& $x_{61} = x_{65} \triangleright x_{59}$ \\
    \hline
    $c_{68}$& $x_{70} = x_{68} \triangleright x_{69}$ \\
    \hline
   $c_{69}$& $x_{69} = x_{71} \triangleright x_{70}$ \\
       \hline
    $c_{70}$& $x_{72} = x_{70} \triangleright x_{71}$ \\
    \hline
    $c_{71}$& $x_{71} = x_{73} \triangleright x_{72}$ \\
    \hline
    $c_{72}$& $x_{74} = x_{72} \triangleright x_{73}$ \\
    \hline
    $c_{73}$& $x_{73} = x_{75} \triangleright x_{74}$ \\
    \hline
    $c_{74}$& $x_{78} = x_{76} \triangleright x_{75}$ \\
    \hline
    $c_{75}$& $x_{81} = x_{76} \triangleright x_{74}$ \\
    \hline
    $c_{76}$& $x_{79} = x_{77} \triangleright x_{75}$ \\
    \hline
    $c_{77}$& $x_{80} = x_{77} \triangleright x_{74}$ \\
    \hline
    $c_{78}$& $x_{75} = x_{82} \triangleright x_{80}$ \\
    \hline
   $c_{79}$& $x_{74} = x_{83} \triangleright x_{80}$ \\
    \hline
    $c_{80}$& $x_{79} = x_{82} \triangleright x_{81}$ \\
    \hline
    $c_{81}$& $x_{78} = x_{83} \triangleright x_{81}$ \\
    \hline
    $c_{82}$& $x_{66} = x_{84} \triangleright x_{78}$ \\
    \hline
    $c_{83}$& $x_{80} = x_{84} \triangleright x_{79}$ \\
    \hline
    $c_{84}$& $x_{27} = x_{85} \triangleright x_{78}$ \\
    \hline
    $c_{85}$& $x_{81} = x_{85} \triangleright x_{79}$ \\
    \hline
\end{tabular}
}
\vskip 1em

By solving the above system of equations using our divide and conquer algorithm, we obtain $13125$ coloring solutions. These consist of 25 solutions having single color for all arcs ($|Im(f)| = 1$), 360 solutions with two distinct colors ($|Im(f)| = 2$), 360 solutions with three distinct colors ($|Im(f)| = 3$), 360 solutions with 21 distinct colors ($|Im(f)| = 21$), 120 solution with 22 distinct colors ($|Im(f)| = 22$), and 1200 solutions with 25 distinct colors ($|Im(f)| = 22$). Consequently, the cardinality of $Hom (Q(L_2),T)$ and corresponding enhanced quandle counting polynomial is given by:
$$|Hom (Q(L_2),T)| = 13125 \text{ and } \phi_E(L2,T) = 25 q + 360 q^2 + 360 q^3 + 360 q^{21} + 120 q^{22} + 1200 q^{25}$$
With this information, we can differentiate $L_2$ from $H$ and also from $L_1$ using both the cardinality of the homomorphism set $Hom (Q(L_2),T)$ and the enhanced counting polynomial associated with this set. Notice that all additional solutions of $L_2$ use $25$ distinct colors, i.e., $|Im(f)| = 25$. $Hom(Q(L_1),T)$ and $Hom(Q(H),T)$ do not exhibit colorings with 25 distinct colors. \\

To derive further insight from the $Hom (Q(L_2),T)$ set, we group arcs with the same color. We find that $Hom (Q(L_2),T)$ has $100$ different coloring schemes that use $25$ colors as shown below. Each of these coloring schemes can have 120 feasible color combinations, making $12000$ additional homomorphisms.
\begin{center}
Coloring scheme with 25 distinct colors (100 configurations):\\

\small{\{\{1, 2\}, \{3, 4, 5, 53, 59\}, \{6, 11, 19, 37\}, \{7, 24, 26, 27, 70\}, \{8, 25, 31\}, \{9, 22, 29, 45\}, \{10, 15, 21, 68\}, \{12, 20, 47\}, \{13, 62, 75\}, \{14, 34\}, \{16, 64\}, \{17, 35, 54, 77, 79\}, \{18, 36, 39, 44, 76\}, \{23, 57\}, \{28, 51, 56, 60\}, \{30, 46, 48, 65\}, \{32\}, \{33,80\}, \{38, 49, 66, 71, 81\}, \{40, 42, 73, 78\}, \{41, 63, 84\}, \{43, 72, 83\}, \textcolor{blue}{\{50, 61, 69, 85\}, \{52, 58, 67\}, \{55, 74, 82\}\}}},\\

. . . . .98 more coloring scheme...\\

\small{\{\{1, 2\}, \{3, 4, 5, 12, 20\},\{6, 52, 58, 77\}, \{7, 24, 26, 27, 38,70\}, \{8, 25, 31, 81\}, \{9, 22, 29\}, \{10, 15, 21, 68\}, \{11, 19, 67\}, \{13, 62, 75\}, \{14, 34\}, \{16, 64\}, \{17, 35, 54\}, \{18, 79\}, \{23, 43, 57\}, \{28, 51, 56, 60\}, \{30, 46, 48, 65\}, \{32, 83\}, \{33, 37, 44, 78, 82\},\{36, 45, 49, 66, 71\}, \{39, 40, 72\}, \{41, 55, 74\}, \{42, 47, 53, 59\}, \textcolor{blue}{\{50,
61, 69, 85\}, \{63, 76, 80, 84\}, \{73\}\}}}\\
Coloring scheme with 22 distinct colors (1 configuration):\\
\small{\{\{1, 2\}, \{3, 4, 5, 10, 21, 47, 51, 60\}, \{6, 67\}, \{7, 26, 27, 28, 46, 48, 68\}, \{8, 25, 29, 45, 49, 66, 69, 85\}, \{9, 22, 50, 61\}, \{11, 19, 52, 58\}, \{12, 20, 42, 53, 59, 82\}, \{13, 17, 54, 62\}, \{14, 55\}, \{15, 56\}, \{16, 43, 57, 83\}, \{18, 37, 63, 77\}, \{23, 64\}, \{24, 65\}, \{30, 38, 70, 81\}, \{31, 36, 39, 71, 76, 80\}, \{32, 40, 72,78\}, \{33, 41, 73, 79\}, \{34, 74\}, \{35, 75\}, \{44, 84\}\}}\\
... similarly for other coloring schemes with 21, 3, and 2 distinct colors
\end{center}

Observe that the coloring schemes with 25 distinct colors have arcs from the right side of $L_2$ (arcs $x_{48}-x_{85}$) that use different color from the set of colors used in the left side of $L_2$(arcs $x_1-x_{45}$) (highlighted in blue). Whereas the coloring scheme that uses 22 colors uses the same set of colors for both the right and the left side of $L_2$. This suggests that coloring schemes with $25$ distinct colors are helping us distinguish $L_1$ from $L_2$. Indeed, when we find the homomorphism set $Hom(Q(L_2),T)$ when $T$ is over $\mathbb{Z}_2$ or $\mathbb{Z}_3$, which does not have 25 distinct colors, we are unable to distinguish $L_2$ from $L_1$ using either the cardinality of homomorphism set or enhanced counting polynomial.
\begin{center}
\resizebox{18cm}{!}{
    \begin{tabular}{|c|c|c|c|c|}
    \hline
    Vector Field &  $\phi_E(L_1,T)$ & $\phi_E(L_2,T)$ & Dist. $L_2$ from $H$? & Dist. $L_2$ from $L_1$ \\
    \hline 
    $Z_2$ & $4 q + 18 q^2 + 6 q^4$ & $4 q + 18 q^2 + 6 q^4$ & Yes & No\\
    \hline
    $Z_3$ & $9 q + 72 q^2 + 24 q^3 + 48 q^9$ & $9 q + 72 q^2 + 24 q^3 + 48 q^9$ & Yes & No\\
    \hline 
    $Z_5$ & $25 q + 360 q^2 + 360 q^3 + 360 q^{21} + 120 q^{22}$  &  $25 q + 360 q^2 + 360 q^3 + 360 q^{21} + 120 q^{22} + 1200 q^{25}$ & Yes & Yes\\
    \hline
\end{tabular}
}
\end{center}
\FloatBarrier

It is also noteworthy that for a given symplectic quandle $T$, if a coloring with $q$ distinct colors exists in $L_1$ then it also exists in $L_2$, but the converse is not true. We will furnish formal proof of this assertion in the next section.\\

\noindent The Mathematica code files to solve equations for Hopf links and Allen-Swenberg links are published with open access (links shown below):\\
\small{\url{https://www.wolframcloud.com/obj/nilesh80/Published/SymplecticHopf2x2-Z5.nb}} \\
\small{\url{https://www.wolframcloud.com/obj/nilesh80/Published/SymplecticAllenSwT1-2x2-Z5.nb}}\\
\small{\url{https://www.wolframcloud.com/obj/nilesh80/Published/SymplecticAllenSwT2-2x2-Z5.nb}}

\subsection{Results with General Allen-Swenberg link}

Our main result that a symplectic quandle $T$ can distinguish all Allen-Swenberg links from the connected sum of Hopf links $H$ is stated in Proposition \ref{ppmains}. We need to establish a few intermediate results to prove Proposition \ref{ppmains}.

\begin{lemma} \label{lmf1eqf2}
For every $f_1 \in  Hom(Q(L_1),T)$ there exist $f_2 \in Hom(Q(L_2),T))$ such that $Im(f_2) = Im(f_1)$
\end{lemma}
\begin{proof}
Take any $f_1:(x_1^1, x_2^1,\cdots,x_{45}^1) \in  Hom(Q(L_1),T)$. We construct $f_2:(x_1^2, x_2^2,\cdots,x_{85}^2)$ such that $f_2 \in Hom(Q(L_2),T))$ and $x_i^2 \in \{x_1^1, x_2^1, \cdots, x_{45}^1\}$ for all $i = 1,2,\cdots,85$.\\

For each $i = 1,2,..,45$ assign $x_i^2 = x_i^1$ i.e. $(x_1^2, x_2^2,\cdots,x_{45}^2) \equiv (x_1^1,x_2^1, \cdots, x_{45}^1)$. \\

Next, we assign the two arcs that connect the right and the left part of $L_2$ as $x_{46}^2 = x_{26}^1$ and $x_{47}^2 = x_{5}^1$. Since $x_3^1 = x_5^1$ and $x_{26}^1 = x_{27}^1$ for all $f \in Hom(Q(L_1),T)$, we have $x_{46}^2 = x_{26}^1 = x_{27}^1$ and $x_{47}^2 = x_{3}^1 = x_{5}^1$. With these assignments equations $(c_1,\cdots, c_{45})$ of $Q(L_2)$ are satisfied because they become identical to $(c_1,\cdots, c_{45})$ of $Q(L_1)$, with corresponding arc values gives by $f_1 \in Hom(Q(L_1),T)$.\\

Next, we construct, $x_{48}^2,\cdots, x_{85}^2$ by assigning each of these variables to one of the variables in $\{x_{1}^2, \cdots x_{45}^2\} \equiv \{x_{1}^1, \cdots x_{45}^1\} $ such that equations $(c_{46},\cdots,c_{85})$ become identical to $(c_6,\cdots,c_{45})$ in the same order. At each step we maintain the previous assignments. The construction process is shown below:
\begin{align*}
&c_{46} (x_{47} = x_{67} \triangleright x_{48}) \equiv  c_{6}(x_5 = x_{6} \triangleright x_{7}) &\text{ if }& x_{48}^2 = x_7^2, x_{67}^2 = x_6^2\\
&c_{47} (x_{49} = x_{48} \triangleright x_{67}) \equiv c_{7}(x_8 = x_{7} \triangleright x_{6}) &&\text{ if } x_{49}^2 = x_8^2\\
& c_{48} (x_{69} = x_{68} \triangleright x_{67}) \equiv c_{8}(x_{29} = x_{28} \triangleright x_{6}) &&\text{ if } x_{69}^2 = x_{29}^2\\
& c_{49} (x_{3} = x_{67} \triangleright x_{68}) \equiv c_{9}(x_{47} = x_{6} \triangleright x_{28}) &&\text{ if } x_{68}^2 = x_{28}^2\\
& c_{50} (x_{50} = x_{49} \triangleright x_{48}) \equiv c_{10}(x_{9} = x_{8} \triangleright x_{7}) &&\text{ if } x_{50}^2 = x_{9}^2\\
&\cdots \cdots && \cdots \\
& c_{83} (x_{80} = x_{84} \triangleright x_{79}) \equiv c_{43}(x_{39} = x_{44} \triangleright x_{41}) &&\text{ if } x_{84}^2 = x_{44}^2\\
& c_{84} (x_{27} = x_{85} \triangleright x_{78}) \equiv c_{44}(x_{46} = x_{45} \triangleright x_{40}) &&\text{ if } x_{85}^2 = x_{45}^2\\
& c_{85} (x_{81} = x_{85} \triangleright x_{79}) \equiv c_{45}(x_{38} = x_{45} \triangleright x_{41}) &&\text{ if } x_{81}^2 = x_{38}^2
\end{align*}
I have not shown few assignments for brevity, but since the crossings on the right side of $L_2$ that is $c_{46}-c_{85}$ is a replica of crossings on the left side $c_6 - c_{45}$, there exists a corresponding map from arcs on the right side to those in the left side when the equations become identical.\\

The above assignments $f_2:(x_1^2,\cdots,x_{85}^2)$ satisfies all equations of $Q(L_2)$, because equations $c_{46}-c_{85}$ is equivalent to $c_6 - c_{45}$ (in the same order), which are already satisfied by $f_2$ in the previous step. Hence, $f_2 \in Hom(Q(L_2),T)$. By construction, $x_i^2 \in \{x_1^1,x_2^1, \cdots, x_{45}^1\}$, for all $i = 1,\cdots, 85$. This implies $Im(f_2) = Im(f_1)$. 
\end{proof}
\vspace{1em}
We can understand this more intuitively using Figure \ref{fig:as12} which shows how $L_1$ and $L_2$ are created from $C$ (shown on the left). 
\FloatBarrier
\begin{figure}[h]
\resizebox{16cm}{!}{
\begin{tabular}{c c c}
    \includegraphics[width = 4cm]{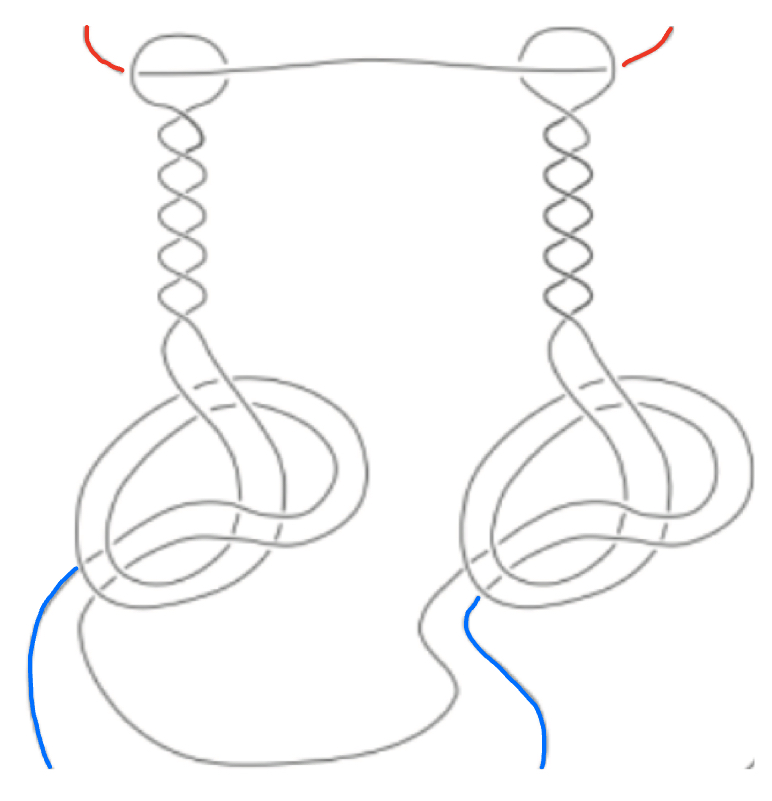} 
	& 
    \includegraphics[width = 6cm]{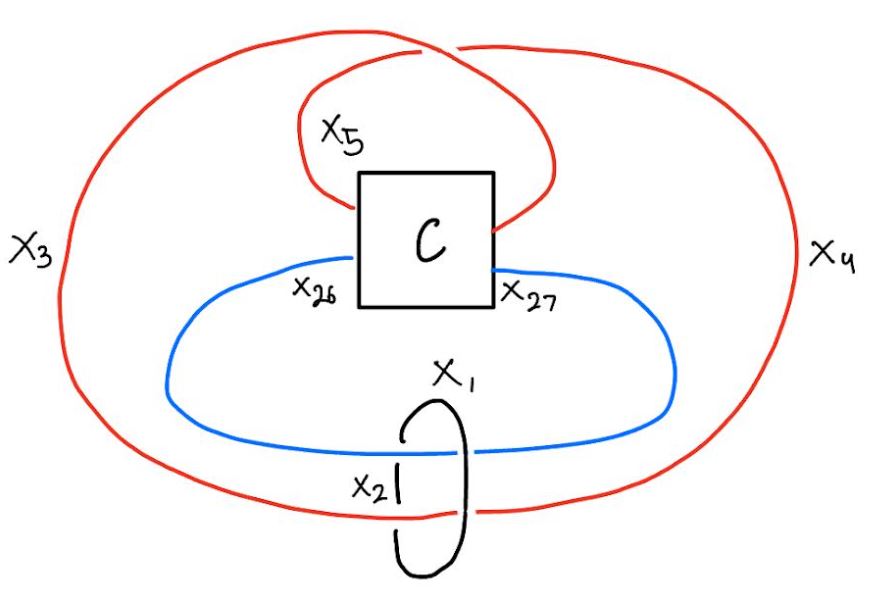} 
    & 
    \includegraphics[width = 6cm]{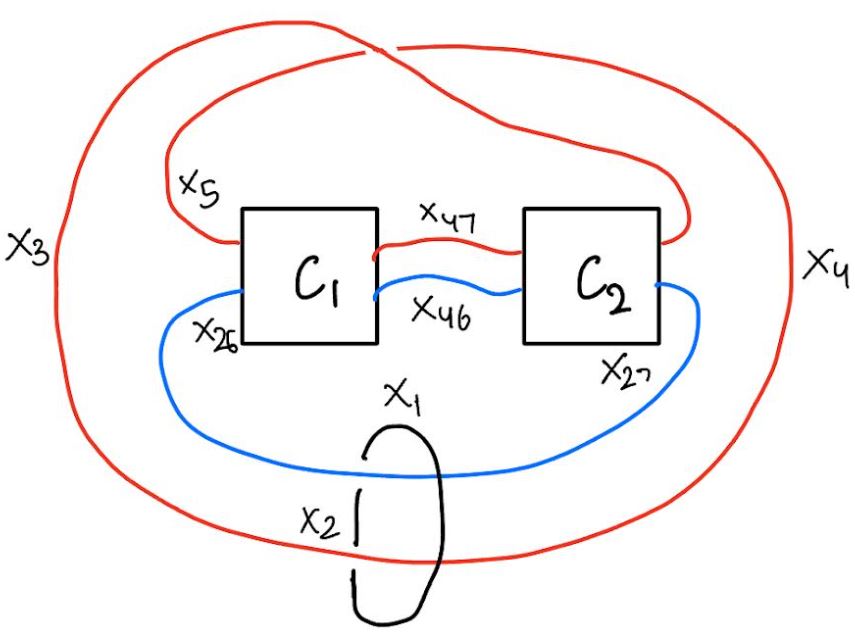} \\
    $C$ & $L_1$ & $L_2$ with $x_{47} = x_5$ and $x_{46} = x_{26}$
\end{tabular}
}
\caption{$L_1$ and $L_2$ generated from $C$ with connection as shown}\label{fig:as12}
\end{figure}
\FloatBarrier
$L_1$ contains a single instance of $C$ whose ends are connected as shown. We know that the coloring space of $L_1$ has $x_1 = x_2$ (shown black), $x_3 = x_4 = x_5$ (shown red) and $x_{26} = x_{27}$ (shown blue). $L_2$ contains two instances of $C$ (call it $C_1$ and $C_2$) that are connected as shown. Lemma \ref{lmf1eqf2} shows that if $L_2$ uses $x_{47} = x_5 = x_3$ (shown red) and $x_{46} = x_{26} = x_{27}$ (shown blue) then all arc variables of $C_2$ can be mapped to arc variables of $C_1$, and all coloring solutions of $L_1$ will satisfy $C_1$. This is intuitive because $C_2$ is a replica of $C_1$, and if both have the same end-points (i.e., same variables) then an identical set of coloring solutions will satisfy their equation sets. Furthermore, by equating $x_{47}$ to $x_5$ and $x_{46}$ to $x_{26}$, the end-points of $C_1$ becomes equivalent to that of $C$ in $L_1$. This implies that all coloring solutions of $L_1$ satisfy $C_1$. Note that $L_2$ may have coloring solutions with $x_{47} \neq x_5$ or $x_{46} \neq x_{26}$.\\

The above result is corroborated by the solutions we got from Mathematica. The configurations that have the same number of colors in $L_1$ and $L_2$ maintain $x_{46} = x_{26}$ and $x_{47} = x_5$, whereas the configurations that are new for $L_2$ (e.g. $|Im(f)| = 25$) do not maintain these relationships. We extend lemma \ref{lmf1eqf2} for any link (denoted as $L_n$) within the Allen-Swenberg infinite series. As elucidated in Figure \ref{fig:asn}, the link $L_n$ is created by replicating $C$ $n$ times, denoted as $C_1,\cdots,C_n$, and interconnecting them as shown. 
%The key point made in the above proof is that arc variables in the left part of $L_2$ can be mapped to the corresponding variables in the right part if the two arcs variable connecting them are equated ($x_{46} = x_{26}$ and $x_{47} = x_5$), and these assignments also ensure that the Left part 
\FloatBarrier
\begin{figure}[h]
\centering
 \includegraphics[width = 9cm]{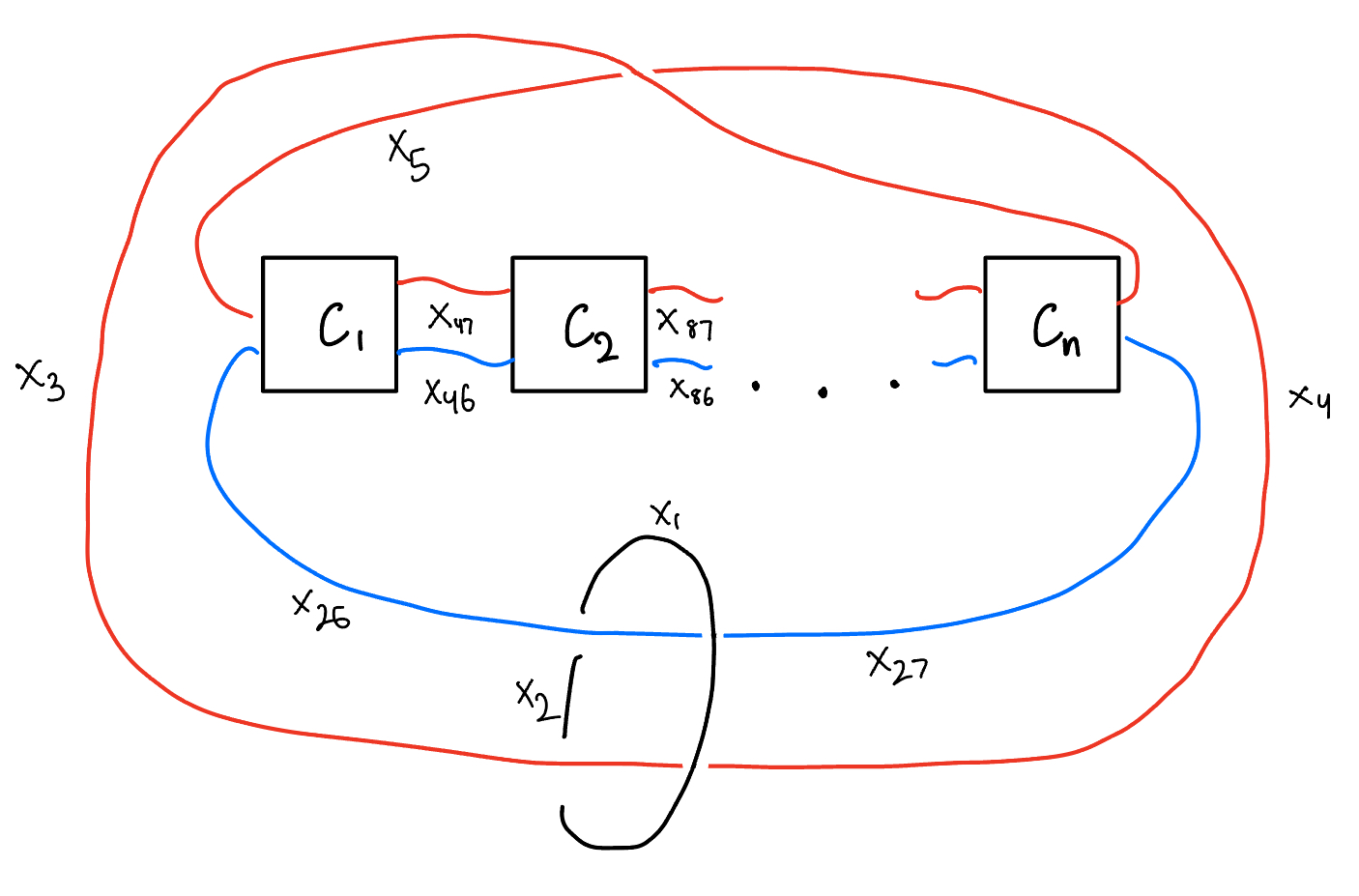}
\caption{Allen-Swenberg link $L_n$ generated by repeating $C$ $n$ times and setting $x_5 = x_{47} = x_{87} = \cdots = x_3$ and $x_{26} = x_{46} = x_{86} = \cdots = x_{27}$}\label{fig:asn}
\end{figure}
\FloatBarrier

\begin{lemma} \label{lmf1eqfn}
For every $f_1 \in  Hom(Q(L_1),T)$ there exist $f_n \in Hom(Q(L_n,T))$ such that $Im(f_n) = Im(f_1)$ for all $n \in \mathbb{Z}_{+}$
\end{lemma}
\begin{proof}
This is true for $n = 2$ from Lemma \ref{lmf1eqf2}. Suppose it is true for $n = k$. This implies that if $x_5 = x_{47} = x_{87} = \cdots = x_3$ and $x_{26} = x_{46} = x_{86} = \cdots = x_{27}$ all arcs in $C_2, C_3, \cdots, C_k$ can be mapped to arcs in $C_1$ whose values are equal to $f_1$. Now we insert another $C$ between $C_1$ and $C_2$ in the above figure, and call this new part $C_2$. Again using Lemma \ref{lmf1eqf2}, all arc variables of $C_2$ can be mapped to $C_1$ if their end points are same i.e. $x_{47} = x_5 = x_{87}$ and $x_{46} = x_{26} = x_{86}$. This insertion does not change any of the arcs of $C_3$ thus preserving the previous solutions for $C_3, C_4,\cdots, C_{k+1}$. This means that all arcs in $C_2, C_3,\cdots, C_{k+1}$ are mapped to arcs of $C_1$ whose arc values are equal to $f_1$. So the lemma is true for $k+1$. By induction, this lemma is true for all $n \in \mathbb{Z}_{+}$.
\end{proof}

%By following exactly the same logic and the fact that $T_k$ for any $k = 1,2,\cdots,n$ is a replica of $T_1$, we can map all arc variables of $T_k$ to that of $T_1$, if we assign the top connecting link (red) the same value as $x_5$ and the bottom connecting link (blue) the same value as $x_{26}$ for each $k$. If all arc variables of $T_1, T_2, ..., T_n$ can be mapped to $T_1$, then we can construct a solution for $L_n$ using the arc variables of $L_1$, which is stated in Lemma \ref{lmf1eqfn}. 

\begin{proposition}\label{ppmains}
\textit{Symplectic quandle $T$ can distinguish $H$ from all links in the infinite series of Allen-Swenberg links.}
\end{proposition}
\begin{proof}
We have demonstrated that the maximum number of colors used in the homomorphism set $Hom(Q(L_1),T)$ is $22$. According to Lemma \ref{lmf1eqfn}, if a solution with $22$ distinct colors exists in $Hom(Q(L_1),T)$ then a solution with the same set of colors must also exist in $Hom(Q(L_n),T)$ for all $n \in \mathbb{Z}_{+}$. This implies that the degree of $\phi_E(L_n, T) \geq 22$ for all $n \in \mathbb{Z}_{+}$. Since the maximum number of colors in $Hom(Q(H),T)$ is $3$, the degree of $\phi_E(H,T)$ is $3$, which implies that the $\phi_E(H,T)$ is different from $\phi_E(L_n,T)$ for all $n \in \mathbb{Z}_{+}$. Hence, $T$ can distinguish $H$ from all links in the Allen-Swenberg series. 
\end{proof}

The next question is whether each link in the infinite series can be differentiated from one another using a symplectic quandle ($T$). We have shown that it is true for the first two links in the series, if $T$ is defined over the field $\mathbb{Z}_{5}$. However,this is not true if $T$ is defined over $\mathbb{Z}_{2}$ or $\mathbb{Z}_{3}$. The larger field size increases the maximum number of colors, thereby enabling the subquandles on the larger links to be colored differently.  Validating this for the third link in the series requires significantly enhanced computing power and memory, as we need to increase the field size. This could be taken up as a future work. Given that $L_2$ is differentiated from $L_1$ based on coloring schemes, in which the arcs in the right part ($C_2$ of $L_2$) use different set of colors from the arcs in the left part ($C_1$ of $L_2$), it is plausible that subsequent links can have different set of colors for each of $C_1, C_2,\cdots,C_n$. Consequently, if the field size is sufficiently large, these links could be differentiated from one another. Based on this observation, I make the following conjecture. 
\begin{conjecture}
\textit{There exists a symplectic quandle $T$, which can distinguish all links in the Allen-Swenberg series from the connected sum of Hopf links, and from each other, using the enhanced quandle counting invariant.}  
\end{conjecture}
\section{Conclusion}
We have clearly established that a symplectic quandle is capable of distinguishing every link within the Allen-Swenberg infinite series from the connected sum of Hopf links, utilizing the \emph{enhanced quandle counting invariant}. Hence, the symplectic quandle paired with the Alexander Conway polynomial is a promising candidate to detect causality in (2+1)-dimensional globally hyperbolic spacetime. We also find that a symplectic quandle in $(\mathbb{Z}_5)^2$ can differentiate between the first two links in the Allen-Swenberg series. Given this evidence, we conjecture that there exists a symplectic quandle that can distinguish all Allen-Swenberg links from each other, which are indistinguishable using Alexander-Conway polynomial.
\section{Acknowledgements}
This project was completed in the Summer of 2023 as a part of the Advanced Mathematics course by Horizon Academic Research Program. The program was supervised by Professor Vladimir Chernov of Dartmouth College and Professor Emanuele Zappala of Yale University. I would like to thank Professor Chernov and Professor Zappala for their constant support in this project.

\bibliography{sc.bib}

\end{document}